\documentclass[11pt]{amsart}
\usepackage[margin=1.4in]{geometry}

\pdfoutput=1

\usepackage[english]{babel}
\usepackage[utf8]{inputenc}
\usepackage[T1]{fontenc}
\usepackage{microtype}
\usepackage{amsmath}
\usepackage{amsfonts}
\usepackage{amssymb}
\usepackage{amsthm}
\usepackage{eucal}
\usepackage{comment}
\usepackage{hyperref}
\hypersetup{colorlinks,allcolors=blue}
\usepackage{tikz-cd}
\tikzcdset{column sep/normal=2.8em}
\tikzcdset{row sep/normal=2.8em}
\usetikzlibrary{babel}
\usepackage{mathrsfs}

\newtheorem{proposition}{Proposition}

\newtheorem{theorem}[proposition]{Theorem}
\newtheorem{corollary}[proposition]{Corollary}

\DeclareMathOperator{\proj}{\mathsf{proj}} 
\DeclareMathOperator{\Hom}{\mathsf{Hom}} 
\DeclareMathOperator{\gldim}{\mathsf{gl.dim}} 
\DeclareMathOperator{\Tor}{\mathsf{Tor}} 
\DeclareMathOperator{\op}{\mathsf{op}} 
\let\mod\relax
\DeclareMathOperator{\mod}{\mathsf{mod}} 
\DeclareMathOperator{\Mod}{\mathsf{Mod}} 

\newcommand{\C}{\mathcal{C}}

\begin{document}

\title{On the equivalence between the existence of $n$-kernels and $n$-cokernels}
\author{Vitor Gulisz}
\address{Mathematics Department, Northeastern University, Boston, MA 02115, USA}
\email{gulisz.v@northeastern.edu}
\author{Wolfgang Rump}
\address{Institute for Algebra and Number Theory, University of Stuttgart, Pfaffenwaldring 57, 70550 Stuttgart, Germany}
\email{rump@mathematik.uni-stuttgart.de}
\date{February 14, 2026}

\begin{abstract}
We give an elementary proof of the statement that if an idempotent complete preadditive category has weak kernels and weak cokernels, then it has $n$-kernels if and only if it has $n$-cokernels, where $n$ is a nonnegative integer. As a consequence, elementary proofs of two results concerning the equality between the global dimensions of certain right and left module categories are obtained.
\end{abstract}

\maketitle

\section*{Introduction}\label{section.1}

The notions of $n$-kernel and $n$-cokernel are key in higher homological algebra, in the context of $n$-abelian and $n$-exact categories, see \cite{MR3519980}. In particular, they are essential to define $n$-almost split sequences, which play a central role in higher \mbox{Auslander--Reiten} \mbox{theory,} see \cite{MR2298819}. Also, $n$-kernels and $n$-cokernels arise in connection with derived equivalences between endomorphism algebras via tilting modules, see \cite[Lemma 3.4]{MR2782196}. Moreover, they can be used to prove the equality between the global dimensions of certain right and left module categories, see Corollaries \ref{corollary.2}, \ref{corollary.3} and \ref{corollary.4}. In fact, these corollaries follow from Theorem \ref{theorem.1}, whose proof is the main subject of this paper. Let us give more details and explain the motivation for our proof.

Let $\C$ be an idempotent complete additive category, and let $n$ be a nonnegative integer. It was proved in \cite[Proposition 6]{MR4971590} that if $\C$ has weak kernels and weak cokernels, then $\C$ has $n$-kernels if and only if $\C$ has $n$-cokernels.\footnote{Actually, in \cite[Proposition 6]{MR4971590}, only the case $n \geqslant 1$ was considered. However, the same argument used in the reference proves the case $n = 0$, due to \cite[Proposition 2.1]{2509.24810}.} The proof presented in the reference is based on the fact that when $\C$ has weak kernels and weak cokernels, the global dimensions of the abelian categories $\mod \C$ and $\mod \C^{\op}$ coincide, where $\mod \C$ and $\mod \C^{\op}$ are the categories of finitely presented right $\C$-modules and finitely presented left $\C$-modules, respectively. This fact, in turn, is far from being trivial, and its standard proof relies on the tensor product $- \otimes -$ and $\Tor_{i}(-,-)$ functors on $\mod \C \times \mod \C^{\op}$, see, for example, \cite[Corollary 5.6]{MR2027559} or \cite[Theorem 67]{MR4971590}. The phenomenon of such an elementary result on the equivalence between the existence of $n$-kernels and $n$-cokernels having such a sophisticated proof intrigued the first author, who then asked the second author for an elementary proof. Within a few hours, the second author sent such a proof for the case $n = 1$ to the first author, who was very surprised and then extended it to the general case $n \geqslant 0$. The purpose of this paper is to present this proof, which holds even by downgrading the assumption that $\C$ is additive to the condition that $\C$ is preadditive, see Theorem \ref{theorem.1}. As a consequence, we obtain an elementary proof of the fact that when $\C$ has weak kernels and weak cokernels, the global dimensions of $\mod \C$ and $\mod \C^{\op}$ coincide, see Corollary \ref{corollary.2}. In particular, our proof shows that, for a coherent ring $\Lambda$, the global dimensions of the categories of finitely presented right $\Lambda$-modules $\mod \Lambda$ and of finitely presented left $\Lambda$-modules $\mod \Lambda^{\op}$ coincide, see Corollary \ref{corollary.3}, which then shows that if $\Lambda$ is noetherian, then the global dimensions of the categories of right $\Lambda$-modules $\Mod \Lambda$ and of left $\Lambda$-modules $\Mod \Lambda^{\op}$ coincide, see Corollary \ref{corollary.4}.

\section*{Preliminaries}\label{section.2}

For the convenience of the reader, before we state and prove the results mentioned above, we briefly recall the definitions used throughout this paper. The reader already familiar with the jargon used in the introduction may skip this section and consult it only when necessary.

A category $\C$ is \textit{preadditive} if the collection of morphisms from an object to another is always an abelian group and the composition of morphisms in $\C$ is bilinear. If $\C$ is preadditive, has a zero object and finite direct sums, then $\C$ is called \textit{additive}. We also say that $\C$ is \textit{idempotent complete} if for every idempotent morphism $e$ in $\C$ there are morphisms $f$ and $g$ in $\C$ such that $e = fg$ and $gf = 1$.

Let $\C$ be a preadditive category, and let $a_{0} : A_{0} \to A_{1}$ be a morphism in $\C$. Recall that a \textit{weak cokernel} of $a_{0}$ is a morphism $a_{1} : A_{1} \to A_{2}$ for which $a_{1}a_{0} = 0$ and such that for every morphism $v : A_{1} \to V$ satisfying $va_{0} = 0$ there is a morphism $w : A_{2} \to V$ such that $v = wa_{1}$. A \textit{cokernel} of $a_{0}$ is a weak cokernel of $a_{0}$ that is an epimorphism. Now, let $n$ be a nonnegative integer. Following \cite{MR3519980}, for $n \geqslant 1$, an \textit{$n$-cokernel} of $a_{0}$ is a sequence of morphisms \[ \begin{tikzcd}
A_{1} \arrow[r, "a_{1}"] & A_{2} \arrow[r, "a_{2}"] & \cdots \arrow[r, "a_{n-1}"] & A_{n} \arrow[r, "a_{n}"] & A_{n+1}
\end{tikzcd} \] such that $a_{i}$ is a weak cokernel of $a_{i-1}$ for each $1 \leqslant i \leqslant n - 1$ and $a_{n}$ is a cokernel of $a_{n-1}$. For $n = 0$, we follow \cite{2509.24810} and define a \textit{$0$-cokernel} of $a_{0}$ to be an epimorphism $g : A_{0} \to X$ for which there is a split monomorphism $f : X \to A_{1}$ such that $a_{0} = fg$. The notions of \textit{weak kernel} and \textit{$n$-kernel} of a morphism in $\C$ are defined dually, by considering the definitions of weak cokernel and $n$-cokernel in $\C^{\op}$, the opposite category of $\C$. Finally, we say that $\C$ has weak cokernels (respectively, weak kernels, $n$-cokernels, $n$-kernels) if every morphism in $\C$ has a weak cokernel (respectively, a weak kernel, an $n$-cokernel, an $n$-kernel).

The reader is referred to \cite{MR4971590} for the definitions of a \textit{finitely presented right $\C$-module} and \textit{finitely presented left $\C$-module}, for an additive category $\C$.

\section*{The results and their proofs}\label{section.3}

As we remarked in the introduction, the following result, Theorem \ref{theorem.1}, was proved in \cite[Proposition 6]{MR4971590} under the additional assumption that $\C$ is an additive category. While the proof presented in \cite[Proposition 6]{MR4971590} depends on the tensor product $- \otimes -$ and $\Tor_{i}(-,-)$ functors on $\mod \C \times \mod \C^{\op}$, the proof that we present below is completely elementary as it only depends on basic notions of categorical algebra.

\begin{theorem}\label{theorem.1}
Let $\C$ be an idempotent complete preadditive category that has weak kernels and weak cokernels, and let $n$ be a nonnegative integer. Then $\C$ has $n$-kernels if and only if $\C$ has $n$-cokernels.
\end{theorem}

\begin{proof}
In what follows, we prove that if $\C$ has $n$-kernels, then $\C$ has $n$-cokernels. Then, by duality (by taking opposite categories), one can deduce the converse. The proof is divided into two cases, namely, $n = 0$ and $n \geqslant 1$. We begin with the latter.

Suppose that $n \geqslant 1$ and that $\C$ has $n$-kernels. Consider a morphism $a_{0} : A_{0} \to A_{1}$ in $\C$. In order to obtain an $n$-cokernel of $a_{0}$, let \[ \begin{tikzcd}
A_{1} \arrow[r, "a_{1}"] & A_{2} \arrow[r, "a_{2}"] & \cdots \arrow[r, "a_{n-1}"] & A_{n} \arrow[r, "a_{n}"] & A_{n+1} \arrow[r, "a_{n+1}"] & A_{n+2}
\end{tikzcd} \] be a sequence of morphisms such that $a_{i}$ is a weak cokernel of $a_{i-1}$ for each $1 \leqslant i \leqslant n + 1$. Moreover, let \[ \begin{tikzcd}
B_{1} \arrow[r, "b_{1}"] & B_{2} \arrow[r, "b_{2}"] & \cdots \arrow[r, "b_{n-1}"] & B_{n} \arrow[r, "b_{n}"] & A_{n+1}
\end{tikzcd} \] be an $n$-kernel of $a_{n+1}$. Because $a_{n+1}a_{n} = 0$ and $b_{n}$ is a weak kernel of $a_{n+1}$, there is a morphism $c_{n} : A_{n} \to B_{n}$ such that $a_{n} = b_{n}c_{n}$. Also, since $b_{n}c_{n}a_{n-1} = a_{n}a_{n-1} = 0$ and $b_{n-1}$ is a weak kernel of $b_{n}$, there is a morphism $c_{n-1} : A_{n-1} \to B_{n-1}$ with $c_{n}a_{n-1} = b_{n-1}c_{n-1}$. By proceeding similarly, we obtain morphisms $c_{i} : A_{i} \to B_{i}$ such that $c_{i+1}a_{i} = b_{i}c_{i}$ for each $1 \leqslant i \leqslant n$, where $c_{n+1}$ is the identity on $A_{n+1}$. \[ \begin{tikzcd}
A_{0} \arrow[r, "a_{0}"] & A_{1} \arrow[r, "a_{1}"] \arrow[d, "c_{1}"'] & A_{2} \arrow[r, "a_{2}"] \arrow[d, "c_{2}"'] \arrow[ld, "d_{2}"', dashed] & \cdots \arrow[r, "a_{n-1}"] \arrow[ld, "d_{3}"', dashed] & A_{n} \arrow[r, "a_{n}"] \arrow[d, "c_{n}"'] \arrow[ld, "d_{n}"', dashed] & A_{n+1} \arrow[r, "a_{n+1}"] \arrow[ld, "d_{n+1}"', dashed, shift right] \arrow[d, equal] & A_{n+2} \arrow[d, equal] \arrow[ld, "d_{n+2}"', dashed] \\
                         & B_{1} \arrow[r, "b_{1}"']                            & B_{2} \arrow[r, "b_{2}"']                                                         & \cdots \arrow[r, "b_{n-1}"']                             & B_{n} \arrow[r, "b_{n}"']                                                         & A_{n+1} \arrow[r, "a_{n+1}"']                                                             & A_{n+2}                                                
\end{tikzcd} \] Furthermore, as $b_{1}c_{1}a_{0} = c_{2}a_{1}a_{0} = 0$ and $b_{1}$ is a monomorphism, $c_{1}a_{0} = 0$. Thus, given that $a_{1}$ is a weak cokernel of $a_{0}$, there is a morphism $d_{2} :A_{2} \to B_{1}$ for which $c_{1} = d_{2}a_{1}$. Also, because $(c_{2} - b_{1}d_{2})a_{1} = c_{2}a_{1} - b_{1}c_{1} = 0$ and $a_{2}$ is a weak cokernel of $a_{1}$, there is a morphism $d_{3} : A_{3} \to B_{2}$ such that $c_{2} - b_{1}d_{2} = d_{3}a_{2}$. By proceeding similarly, we get morphisms $d_{i+1} : A_{i+1} \to B_{i}$ satisfying that $c_{i} - b_{i-1}d_{i} = d_{i+1}a_{i}$ for each $2 \leqslant i \leqslant n + 1$.

Next, note that $0 = d_{n+2}a_{n+1}b_{n} = (1 - b_{n}d_{n+1})b_{n} = b_{n} - b_{n}d_{n+1}b_{n}$, so that $b_{n} = b_{n}d_{n+1}b_{n}$. Consequently, $(d_{n+1}b_{n})^{2} = d_{n+1}b_{n}$, that is, $d_{n+1}b_{n}$ is idempotent. In this case, given that $\C$ is idempotent complete, there are morphisms $f : C \to B_{n}$ and $g : B_{n} \to C$ for which $fg = d_{n+1}b_{n}$ and $gf = 1$. We claim that $gc_{n}$ is a cokernel of $a_{n-1}$. In fact, first, observe that $fgc_{n}a_{n-1} = d_{n+1}b_{n}c_{n}a_{n-1} = d_{n+1}a_{n}a_{n-1} = 0$, and because $f$ is a monomorphism, $gc_{n}a_{n-1} = 0$. Next, suppose that $v : A_{n} \to V$ is a morphism such that $va_{n-1} = 0$. Then, as $a_{n}$ is a weak cokernel of $a_{n-1}$, there is a morphism $w : A_{n+1} \to V$ for which $v = wa_{n}$. Consequently, $v = wa_{n} = wb_{n}c_{n} = wb_{n}d_{n+1}b_{n}c_{n} = wb_{n}fgc_{n}$. Therefore, $gc_{n}$ is a weak cokernel of $a_{n-1}$. Thus, to conclude that $gc_{n}$ is a cokernel of $a_{n-1}$, it suffices to show that $gc_{n}$ is an epimorphism. Well, suppose that $x : C \to X$ is a morphism with $xgc_{n} = 0$. Because $gf = 1$ and $fg = d_{n+1}b_{n}$, we have $g = gfg = gd_{n+1}b_{n}$, hence $0 = xgc_{n} = xgd_{n+1}b_{n}c_{n} = xgd_{n+1}a_{n}$. Therefore, since $a_{n+1}$ is a weak cokernel of $a_{n}$, there is a morphism $z : A_{n+2} \to X$ such that $xgd_{n+1} = za_{n+1}$. Then $xg = xgd_{n+1}b_{n} = za_{n+1}b_{n} = 0$, which implies that $x = 0$ as $g$ is an epimorphism. Finally, since $gc_{n}$ is a cokernel of $a_{n-1}$ and $a_{i}$ is a weak cokernel of $a_{i-1}$ for each $1 \leqslant i \leqslant n - 1$, we conclude that \[ \begin{tikzcd}
A_{1} \arrow[r, "a_{1}"] & A_{2} \arrow[r, "a_{2}"] & \cdots \arrow[r, "a_{n-1}"] & A_{n} \arrow[r, "gc_{n}"] & C
\end{tikzcd} \] is an $n$-cokernel of $a_{0}$. Hence $\C$ has $n$-cokernels.\footnote{We remark that the above arguments can be simplified when $n = 1$. Indeed, in this case, from $1 - b_{1}d_{2} = d_{3}a_{2}$, it follows that $b_{1}(1 - d_{2}b_{1}) = (1 - b_{1}d_{2})b_{1} = d_{3}a_{2} b_{1} = 0$, hence $1 - d_{2}b_{1} = 0$ as $b_{1}$ is a monomorphism, so that $d_{2}b_{1} = 1$. Therefore, in the proof, we can take both $f$ and $g$ to be the identity on $B_{1}$, and we deduce that $c_{1}$ is a cokernel of $a_{0}$.}

Now, we consider the case $n = 0$. Assume that $\C$ has $0$-kernels, and let $a_{0} : A_{0} \to A_{1}$ be a morphism in $\C$. In order to obtain a $0$-cokernel of $a_{0}$, let $a_{1} : A_{1} \to A_{2}$ be a weak cokernel of $a_{0}$, and take a $0$-kernel $b : B \to A_{2}$ of $a_{1}$, so that $b$ is a monomorphism and there is a split epimorphism $c : A_{1} \to B$ for which $a_{1} = bc$. Note that $bca_{0} = a_{1}a_{0} = 0$, hence $ca_{0} = 0$ as $b$ is a monomorphism. \[ \begin{tikzcd}
A_{0} \arrow[r, "a_{0}"] & A_{1} \arrow[r, "a_{1}"] \arrow[d, "c"'] & A_{2} \arrow[d, equal] \\
                         & B \arrow[r, "b"']                        & A_{2}                 
\end{tikzcd} \] Let $r : B \to A_{1}$ be such that $cr = 1$. Then $(rc)^{2} = rc$, which implies that $(1 - rc)^{2} = 1 - rc$, that is, $1 - rc$ is idempotent. Since $\C$ is idempotent complete, there are morphisms $f : C \to A_{1}$ and $g : A_{1} \to C$ satisfying $fg = 1 - rc$ and $gf = 1$. We claim that $ga_{0}$ is a $0$-cokernel of $a_{0}$. Indeed, observe that $fga_{0} = (1 - rc)a_{0} = a_{0}$. Therefore, given that $f$ is a split monomorphism, it suffices to show that $ga_{0}$ is an epimorphism to conclude that $ga_{0}$ is a $0$-cokernel of $a_{0}$. To verify this, suppose that $x : C \to X$ is a morphism for which $xga_{0} = 0$. Given that $a_{1}$ is a weak cokernel of $a_{0}$, there is a morphism $z : A_{2} \to X$ such that $xg = za_{1}$. Then $x = xgf = za_{1}f = zbcf$. However, from $1 = rc + fg$, we get that $f = rcf + f$, so that $rcf = 0$, which implies that $cf = 0$ as $r$ is a monomorphism. Consequently, $x = 0$, and $ga_{0}$ is a $0$-cokernel of $a_{0}$. Hence $\C$ has $0$-cokernels.
\end{proof}

We can now use Theorem \ref{theorem.1} to deduce Corollary \ref{corollary.2}, which was used in \cite[Proposition 6]{MR4971590} to prove Theorem \ref{theorem.1} in the case that $\C$ is an additive category. By doing so, we obtain an elementary proof of Corollary \ref{corollary.2}, whose standard proof can be found, for example, in \cite[Corollary 5.6]{MR2027559} or \cite[Theorem 67]{MR4971590}.

\begin{corollary}\label{corollary.2}
Let $\C$ be an idempotent complete additive category that has weak kernels and weak cokernels. Then $\gldim (\mod \C) = \gldim (\mod \C^{\op})$.
\end{corollary}

\begin{proof}
To begin with, recall that the categories $\mod \C$ and $\mod \C^{\op}$ are abelian since $\C$ has weak kernels and weak cokernels, see \cite[page 41]{auslander1971representation} or \cite[Corollary 1.5]{MR0209333}, hence it makes sense to consider their global dimensions. To verify that these dimensions coincide, it is enough to show that $\gldim (\mod \C) \leqslant m$ if and only if $\gldim (\mod \C^{\op}) \leqslant m$, whenever $m$ is a nonnegative integer. So, let $m$ be such an integer.

Suppose that $m \geqslant 1$, and write $m = n + 1$, where $n$ is a nonnegative integer. By \cite[Proposition 2.1]{2509.24810} and \cite[Proposition 5]{MR4971590}, it holds that $\gldim (\mod \C) \leqslant m$ if and only if $\C$ has $n$-kernels. Dually, it holds that $\gldim (\mod \C^{\op}) \leqslant m$ if and only if $\C$ has $n$-cokernels. Thus, it follows from Theorem \ref{theorem.1} that $\gldim (\mod \C) \leqslant m$ if and only if $\gldim (\mod \C^{\op}) \leqslant m$.

To complete the proof, it remains to show that $\gldim (\mod \C) = 0$ if and only if $\gldim (\mod \C^{\op}) = 0$. Well, it is easy to see that $\gldim (\mod \C) = 0$ if and only if every morphism $f$ in $\C$ can be written as $f = gh$ in $\C$, where $h$ is a split epimorphism and $g$ is a split monomorphism. By replacing $\C$ by $\C^{\op}$ in this statement, and by noticing that its latter condition does not change, we then conclude that $\gldim (\mod \C) = 0$ if and only if $\gldim (\mod \C^{\op}) = 0$.
\end{proof}

The following well-known result is a particular case of Corollary \ref{corollary.2}.

\begin{corollary}\label{corollary.3}
Let $\Lambda$ be a coherent ring. Then $\gldim (\mod \Lambda) = \gldim (\mod \Lambda^{\op})$.
\end{corollary}

\begin{proof}
Let $\proj \Lambda$ be the category of finitely generated projective right $\Lambda$-modules, which is idempotent complete and additive. Then there are equivalences of categories $\mod (\proj \Lambda) \approx \mod \Lambda$ and $\mod (\proj \Lambda)^{\op} \approx \mod \Lambda^{\op}$, which are given by the evaluation at $\Lambda$, see \cite[Proposition 2.7]{MR349747} or \cite[Section 8]{MR4971590}. Since $\Lambda$ is coherent, the categories $\mod \Lambda$ and $\mod \Lambda^{\op}$ are abelian, hence so are $\mod (\proj \Lambda)$ and $\mod (\proj \Lambda)^{\op}$. Consequently, $\proj \Lambda$ has weak kernels and weak cokernels, see \cite[page 41]{auslander1971representation} or \cite[Corollary 1.5]{MR0209333}. Therefore, by Corollary \ref{corollary.2} and the previous equivalences, $\gldim (\mod \Lambda) = \gldim (\mod \Lambda^{\op})$.
\end{proof}

We remark that, by following the arguments employed in the proofs of Theorem \ref{theorem.1} and Corollary \ref{corollary.2}, it is also possible to prove Corollary \ref{corollary.3} directly, in terms of projective resolutions of objects in $\mod \Lambda$ and $\mod \Lambda^{\op}$. In fact, $n$-kernels and $n$-cokernels in $\C$ correspond to certain projective resolutions in $\mod \C$ and $\mod \C^{\op}$, respectively, see \cite[page 1131]{MR4971590} and \cite[page 5]{2509.24810}. Thus, one could reformulate the proofs of Theorem \ref{theorem.1} and Corollary \ref{corollary.2} in terms of finitely presented right and left $\C$-modules, and then reproduce them for the case of finitely presented right and left $\Lambda$-modules.\footnote{Note that such proofs for $\C$-modules would rely on the functors $(-)^{\ast} : \mod \C \leftrightarrow \mod \C^{\op}$ described in \cite[page 1132]{MR4971590}. Hence the proof for $\Lambda$-modules would make use of the functors $(-)^{\ast} : \mod \Lambda \leftrightarrow \mod \Lambda^{\op}$ given by $(-)^{\ast} = \Hom_{\Lambda}(-,\Lambda)$ in $\mod \Lambda$ and by $(-)^{\ast} = \Hom_{\Lambda^{\op}}(-,\Lambda^{\op})$ in $\mod \Lambda^{\op}$.}

We also observe that, for a coherent ring $\Lambda$, the global dimensions of $\mod \Lambda$ and of $\mod \Lambda^{\op}$ coincide with the weak global dimension of $\Lambda$, see \cite[Proposition 1.1]{MR306265}. This dimension is usually used to prove the well-known Corollary \ref{corollary.4}, see \cite[Corollary 5]{MR74406}, and depends on the tensor product $- \otimes_{\Lambda} -$ and $\Tor_{i}^{\Lambda}(-,-)$ functors on $\Mod \Lambda \times \Mod \Lambda^{\op}$. The proof that we present below for Corollary \ref{corollary.4}, however, bypasses the introduction of these functors.

\begin{corollary}\label{corollary.4}
Let $\Lambda$ be a noetherian ring. Then $\gldim (\Mod \Lambda) = \gldim (\Mod \Lambda^{\op})$.
\end{corollary}

\begin{proof}
Given that $\Lambda$ is noetherian, the categories $\mod \Lambda$ and $\mod \Lambda^{\op}$ coincide with the categories of finitely generated right $\Lambda$-modules and finitely generated left $\Lambda$-modules, respectively. Therefore, it follows from a result of Auslander, namely, \cite[Theorem 1]{MR74406}, that $\gldim (\Mod \Lambda) = \gldim (\mod \Lambda)$ and $\gldim (\Mod \Lambda^{\op}) = \gldim (\mod \Lambda^{\op})$. Thus, we conclude from Corollary \ref{corollary.3} that $\gldim (\Mod \Lambda) = \gldim (\Mod \Lambda^{\op})$.
\end{proof}


\begin{thebibliography}{1}

\bibitem{MR74406}
Maurice Auslander.
\newblock On the dimension of modules and algebras. {III}. {G}lobal dimension.
\newblock {\em Nagoya Math. J.}, 9:67--77, 1955.

\bibitem{auslander1971representation}
Maurice Auslander.
\newblock Representation dimension of artin algebras.
\newblock {\em Lecture notes, Queen Mary College, London}, 1971.

\bibitem{MR349747}
Maurice Auslander.
\newblock Representation theory of {A}rtin algebras {I}.
\newblock {\em Comm. Algebra}, 1:177--268, 1974.

\bibitem{MR2027559}
Apostolos Beligiannis.
\newblock On the {F}reyd categories of an additive category.
\newblock {\em Homology Homotopy Appl.}, 2:147--185, 2000.

\bibitem{MR0209333}
Peter Freyd.
\newblock Representations in abelian categories.
\newblock In {\em Proc. {C}onf. {C}ategorical {A}lgebra ({L}a {J}olla, {C}alif., 1965)}, pages 95--120. Springer-Verlag New York, Inc., New York, 1966.

\bibitem{2509.24810}
Vitor Gulisz.
\newblock A functorial approach to $0$-abelian categories, 2025.
\newblock \mbox{arXiv:} 2509.24810.

\bibitem{MR4971590}
Vitor Gulisz.
\newblock A functorial approach to {$n$}-abelian categories.
\newblock {\em C. R. Math. Acad. Sci. Paris}, 363:1123--1175, 2025.

\bibitem{MR2782196}
Wei Hu and Changchang Xi.
\newblock {$\mathscr{D}$}-split sequences and derived equivalences.
\newblock {\em Adv. Math.}, 227(1):292--318, 2011.

\bibitem{MR2298819}
Osamu Iyama.
\newblock Higher-dimensional {A}uslander-{R}eiten theory on maximal orthogonal subcategories.
\newblock {\em Adv. Math.}, 210(1):22--50, 2007.

\bibitem{MR3519980}
Gustavo Jasso.
\newblock {$n$}-Abelian and {$n$}-exact categories.
\newblock {\em Math. Z.}, 283(3-4):703--759, 2016.

\bibitem{MR306265}
D.~George McRae.
\newblock Homological dimensions of finitely presented modules.
\newblock {\em Math. Scand.}, 28:70--76, 1971.

\end{thebibliography}
\end{document}